\documentclass[11pt]{article}
\usepackage[hmargin=1.2in,vmargin=1.33in]{geometry}

\usepackage{amsmath, amsthm, amssymb, verbatim}
\usepackage{color}
\usepackage[colorlinks]{hyperref}
\usepackage{enumerate}

\newif\ifShowTodo\ShowTodofalse
\ifShowTodo
\usepackage[colorinlistoftodos, textwidth=4cm]{todonotes}
\usepackage[displaymath, tightpage]{preview}

\newcommand{\todoInl}[1]{{\todo[inline, color=red!70]{#1}}}
\else
\usepackage[disable]{todonotes}
\newcommand{\todoInl}[1]{}
\fi
\newcommand{\note}[1]{\todo[color=yellow!40]{#1}}
\newcommand{\question}[1]{\todo[color=blue!40]{#1}}

\newtheorem{theorem}{Theorem}[section]
\newtheorem{definition}[theorem]{Definition} \newtheorem{lemma}[theorem]{Lemma}
\newtheorem{proposition}[theorem]{Proposition}
\newtheorem{corollary}[theorem]{Corollary}

 \renewcommand{\P}{\operatorname*{\mathbf{P}}}
\newcommand{\E}{\operatorname*{\mathbf{E}}}
\newcommand{\Var}{\operatorname*{\mathbf{Var}}}

\DeclareMathOperator{\length}{length}
\DeclareMathOperator{\Dist}{{\bf D}}

\DeclareMathOperator{\CONST}{CONST} \DeclareMathOperator{\DICT}{DICT}
\DeclareMathOperator{\NONMANIP}{NONMANIP} 
 \DeclareMathOperator{\Inf}{Inf}

\newcommand{\adj}[2]{\ensuremath{[#1:#2]}}
\newcommand{\adji}[3]{\ensuremath{\adj{#1}{#2}_{#3}\,}}

\newcommand{\I}{\ensuremath{\mathrm{I}}}

\newcommand{\wt}{\widetilde} 

\newcommand{\poly}{\ensuremath{\mathrm{poly}}}  

\definecolor{Red}{rgb}{1,0,0} \definecolor{Blue}{rgb}{0,0,1}
\definecolor{Olive}{rgb}{0.41,0.55,0.13} \definecolor{Green}{rgb}{0,1,0}
\definecolor{MGreen}{rgb}{0,0.8,0} \definecolor{DGreen}{rgb}{0,0.55,0}
\definecolor{Yellow}{rgb}{1,1,0} \definecolor{Cyan}{rgb}{0,1,1}
\definecolor{Magenta}{rgb}{1,0,1} \definecolor{Orange}{rgb}{1,.5,0}
\definecolor{Violet}{rgb}{.5,0,.5} \definecolor{Purple}{rgb}{.75,0,.25}
\definecolor{Brown}{rgb}{.75,.5,.25} \definecolor{Grey}{rgb}{.5,.5,.5}
\definecolor{Black}{rgb}{0,0,0}

\newcommand{\eps}{\epsilon}
\newcommand{\plur}{\mathbf{pl}}
\newcommand{\remove}[1]{}
\newcommand{\inverse}[1]{f^{-1}({#1})}

\begin{document}

\title{The Geometry of Manipulation - a Quantitative Proof of the Gibbard Satterthwaite Theorem}

\author{
  Marcus Isaksson
  \thanks{
    Chalmers University of Technology and G{\"o}teborg University,
    SE-41296 G{\"o}teborg, Sweden.
    maris@chalmers.se.
  }
  \and
  Guy Kindler\thanks{Incumbent of the Harry and Abe Sherman Lectureship Chair at the Hebrew Univeristy of Jerusalem. Supported by the Israel Science Foundation and by the Binational Science Foundation. }
  \and
  Elchanan Mossel
  \thanks{Weizmann Institute and U.C. Berkeley
    mossel@stat.berkeley.edu.
    Weizmann Institute of Science and U.C. Berkeley. Supported by DMS 0548249 (CAREER) award, by ISF grant 1300/08, by a  Minerva Foundation grant and by an ERC Marie Curie Grant 2008 239317. 
  } }

\maketitle

\begin{abstract}
  We prove a quantitative version of the Gibbard-Satterthwaite
  theorem.  We show that a uniformly chosen voter profile for a
  neutral social choice function $f$ of $q \geq 4$ alternatives and
  $n$ voters will be manipulable with probability at least $10^{-4}
  \eps^2 n^{-3} q^{-30}$, where $\eps$ is the minimal statistical
  distance between $f$ and the family of dictator functions.

 Our results extend those of \cite{FrKaNi:08}, which were obtained
    for the case of $3$ alternatives, and imply that the approach of
    masking manipulations behind computational hardness (as considered
    in
    \cite{BarthOrline:91,ConitzerS03b,ElkindL05,ProcacciaR06,ConitzerS06})
    cannot hide manipulations completely.

  Our proof is geometric. More specifically it extends the method of
  canonical paths to show that the measure of the profiles that lie on
  the interface of $3$ or more outcomes is large. To the best of our
  knowledge our result is the first isoperimetric result to establish
  interface of more than two bodies.
\end{abstract}

\newpage

\section{Introduction}

Social choice theory studies methods of collective decision making,
and their interplay with social welfare and individual preference and
behavior. Rigorous study of social choice dates back to the
18'{th} century, when Condorcet discovered the following voting
paradox: in a social ranking of three alternatives that is determined
by the majority vote, an `irrational' circular ranking may occur where
a candidate $A$ is preferred over a candidate $B$, $B$ is preferred
over $C$, and $C$ is preferred over $A$. Social choice
theory in its modern form was established in the 1950's with the discovery of Arrow's
impossibility theorem~\cite{Arrow:50,Arrow:63}, which showed that all social ranking systems
that satisfy a few reasonable conditions must either obtain irrational
circular outcomes, or be dictatorships (a dictatorship is a system
where the ranking is determined by just one voter).

\paragraph{Manipulations.} Many of the results in the study of social
choice are negative, showing that certain desired properties of social
choice schemes cannot be attained.
One of the hallmark examples of such theorems was proved by
Gibbard and Satterthwaite~\cite{Gibbard:73,Satterthwaite:75}. Their theorem considers a
voting system where each of $n$ voters rank $q$ alternatives, and the
winner is determined according to some pre-defined \emph{social choice
  function} $f \colon L_q^n \rightarrow [q]$ of all the voters'
rankings---here $L_q$ denotes the set of total orderings of the $q$
alternatives.

We say that a social choice function is {\em manipulable}, if a
situation may occur where a voter who knows the rankings given by
other voters can change her own ranking in a way that does not reflect
her true preferences, but which leads to an outcome that is more
desirable to her. Formally

\begin{definition}[Manipulation point]
  For a ranking $x \in L_q$, write $a \stackrel{x}{>} b$ to denote
  that the alternative $a$ is preferred by $x$ over $b$. A social
  choice function $f \colon L_q^n \rightarrow [q]$ is
  \emph{manipulable} at $x \in L_q^n$ if there exist a $y \in L_q^n$
  and $i\in [n]$ such that $x$ and $y$ only differ in the $i$'th
  coordinate and
  \begin{equation}
    f(y) \stackrel{x_i}{>} f(x)
  \end{equation}
  In this case we also say that $x$ is a {\em manipulation point} of $f$,
  and that $(x,y)$ is {\em a manipulation pair} for $f$.
  We say that $f$ is {\em manipulable}, if it is manipulable at some point $x$.
  We also say that $x$ is an $r$-{\em manipulation point} of $f$,
  if $f$ has a manipulation pair $(x,y)$ such that $y$ is obtained from
  $x$ by permuting (at most) $r$ adjacent alternatives in one of the
  coordinates of $x$.
\end{definition}

Gibbard and Satterthwaite proved that any social
choice function which attains three or more values, and whose outcome
does not depend on just one voter, must be manipulable.

\begin{theorem}[Gibbard-Satterthwaite~\cite{Gibbard:73,Satterthwaite:75}]
  \label{thm:GS}
  Any social choice function $f \colon L_q^n \rightarrow [q]$ which takes at
  least three values and is not a dictator is manipulable.
\end{theorem}

The Gibbard-Satterthwaite theorem has contributed significantly to the realization that it is unlikely to expect truthfulness
in the context of voting. In a way, this and other results in social choice theory, contributed to the
development of mechanism design, a field centered around
developing social mechanisms that obtain desirable results even when
each member of the society acts selfishly.

\paragraph{Quantitative social choice.}
Theorem~\ref{thm:GS} is tight in the sense that \emph{monotone} social choice
functions which are dictators or only have two possible outcomes
are indeed non-manipulable (a function is non-monotone, and clearly manipulable, if for some set of rankings a voter can change the outcome from say $a$ to $b$ by moving $a$ ahead of $b$ in his preference).
\question{Define monotone better?}
It is interesting, however, to study manipulation quantitatively, asking not
just whether a function is manipulable but how many manipulations
occur in it.
\note{Removed def. of $\P$ and $\E$ since it wasn't used that way throughout.}
To state results in quantitative social choice we need
to define the distance between social choice functions.

\begin{definition}[Distance between social choice functions]
The
distance $\Dist(f,g)$ between two social choice functions $f,g \colon L_q^n
\rightarrow [q]$ is defined as the fraction of inputs on which they
differ: $\Dist(f,g) = \P[f(X) \neq g(X)]$, where $X \in L_q^n$ is
uniformly selected.
For a class $G$ of social functions, we write
$\Dist(f,G) = \min_{g \in G} \Dist(f,g)$.
\end{definition}

We also define some classes of functions that may not have any
manipulation points.

\begin{definition}
  We use the following three classes of functions, defined for
  parameters $n$ and $q$ that remain implicit (when used, the
  parameters will be obvious from the context):
\begin{align*}
  \CONST &= \{f \colon L_q^n \rightarrow [q] \mid f \text{ is constant }\}
  \\
  \DICT_i &= \{f \colon L_q^n \rightarrow [q] \mid f \text{ only depend on the
    $i$:th coordinate }\} \text{ , for } i \in [n]
  \\
  \DICT &= \cup_{i=1}^n \DICT_i
  \\
  \NONMANIP &= \{f \colon L_q^n \rightarrow [q] \mid f \text{ is either a
    dictator or takes at most two values}\}
\end{align*}
\question{$\NONMANIP$ is bad name!}
\end{definition}

\subsection{Our results}

Our results only apply to social choice functions which are {\em
  neutral}. A social choice function is neutral if it is invariant under
changes made to the names of the alternatives (see
Definition~\ref{def:neutrality} for a formal description). In our
first main result we show the following lower bound on the number of
manipulation points in a neutral social function:
\begin{theorem}
  \label{thm:neutralBound}
  Fix $q \ge 4$ and
  let $f \colon L_q^n \rightarrow [q]$ be a neutral social choice function
  with $\Dist(f, \DICT) \ge \eps$.
  \question{DICT contains manipulable functions. Can we restrict to only monotone dictators which should be the \emph{right} set to compare with.}
  Then,
  \begin{equation}
    \P(f\text{ is manipulable at } X)
    \ge
    \frac{\eps^2}{2n^3q^6(q!)^2}
  \end{equation}
  where $X\in L_q^n$ is selected uniformly.
\end{theorem}

Note that the result above directly implies the following:
\begin{corollary} \label{cor:NeutralBound}
Fix $q \ge 4$ and
  let $f \colon L_q^n \rightarrow [q]$ be a neutral social choice function
  with $\Dist(f, \DICT) \ge \eps$.
  Then,
\[
\P((X,Y) \text{ is a manipulable pair for } f) \ge \frac{\eps^2}{2n^4 q^6(q!)^3},
\]
 where $X\in L_q^n$ is selected uniformly, and $Y$ is obtained from
 $X$ by uniformly selecting a coordinate $i\in\{1,..,n\}$ and
  resetting the $i$'th coordinate to a random preference.
\end{corollary}

The result above has super exponential dependency on the number of alternatives
$q$. A more refined analysis yields the following theorem.

\begin{theorem}[main theorem]
  \label{thm:refNeutralBound}
  Fix $q \ge 4$ and
  let $f \colon L_q^n \rightarrow [q]$ be a neutral social choice function
  with $\Dist(f, \DICT) \ge \eps$.
  Then,
  \begin{equation}
    \P(f\text{ is manipulable at } X)
\geq     \P(X\text{ is a } 4\text {-manipulation point of } f) \geq
    \frac{\eps^2}{10^4 n^3 q^{30}}
  \end{equation}
  where $X\in L_q^n$ is uniformly selected.
\end{theorem}

A result similar to Theorem~\ref{thm:refNeutralBound} was obtained for
the case $q=3$ in~\cite{FrKaNi:08}, but the result
of~\cite{FrKaNi:08} counted manipulation pairs rather than
manipulation points. Translating the bound on
the fraction of manipulation points in Theorem~\ref{thm:refNeutralBound}
directly to the case of pairs deteriorates the lower bound, inserting
a factor of $q!$ in the denominator.
However using the stronger bound on the fraction of 4-manipulation points,
a direct corollary lower bounds the fraction of
manipulation pairs of a certain kind while keeping the polynomial
dependency on $q$.
\begin{corollary}[manipulation pairs]
 \label{cor:refNeutralBound}
 Fix $q \ge 4$ and
 let $f \colon L_q^n \rightarrow [q]$ be a neutral social choice function
 with $\Dist(f, \DICT) \ge \eps$.
 Then,
 \begin{equation}
   \P((X,Y)\text{ is a manipulation pair for }f)
   \ge
   \frac{\eps^2}{10^{9} n^4 q^{34}}
 \end{equation}
 where $X\in L_q^n$ is uniformly selected, and $Y$ is obtained from
 $X$ by uniformly selecting a coordinate $i\in\{1,..,n\}$, then selecting $4$ adjacent alternatives in $X_i$ and
 randomly permuting them.
\end{corollary}

The case of large $q$, solved here, was left as the main open problem
in~\cite{FrKaNi:08}. Their main motivation was that deriving
quantitative versions of Gibbard-Satterthwaite theorems with
polynomial dependency of $q$ and $n$ would indicate that from the
computational complexity point of view it is easy on average to find
manipulation points. This point is discussed in more detail in
the related work subsection.

\medskip\noindent
Our lower bound for the number of manipulation points deteriorates
polynomially with the number of voters, $n$, and the number $q$ of
alternatives. Some polynomial deterioration as a
function of $n$ is necessary. This can be observed by considering the
plurality function $\plur:L_q^n\rightarrow[q]$, whose value is defined
to be the candidate which is top ranked by the largest number of
voters (break ties by picking the candidate which is top ranked by the
'leftmost' voter). It is easy to observe that a point where no ties
are formed is not a manipulation point of $\plur$, and that for any
fixed $q$ the fraction of points that do contain ties is polynomially
small in $n$. As for the dependency on $q$---we do not know whether it is necessary.

\subsection{History and related work}

The Gibbard-Satterthwaite theorem presented a difficulty in designing
social choice functions, namely that of strategic voting. A line of
research aimed at overcoming these difficulties suggested
constructions of social choice functions where it is computationally
difficult for a voter to find beneficial
manipulation~\cite{BaToTr:89,BarthOrline:91,ConitzerS03b,ElkindL05}.
However these constructions considered worst case analysis---they did
not rule out the possibility that \emph{on average}, finding a
manipulation may be easy. Indeed, some results showed that
finding manipulations is easy on average for certain restricted classes of
social choice functions~\cite{ProcacciaR06,ConitzerS06, Kelly93} (see
also the survey~\cite{procaccia-survey}).

Recently, a result of Friedgut, Kalai and Nisan~\cite{FrKaNi:08}
provided a very general result, showing that in the case of a neutral
social choice function between $3$ alternatives even a random
attempted manipulation is beneficial for a voter with non-negligible
probability.
 Adapted to
our notation, the main result of~\cite{FrKaNi:08} can be stated as
follows:
\begin{theorem}[\cite{FrKaNi:08}]
  \label{thm:FrKaNi}
  There exists a constant $C>0$ with the following property. Let $f
  \colon L_3^n \rightarrow [3]$ be a neutral social choice function
  with $\Dist(f, \DICT) \ge \eps$.
  Then,
  \begin{equation}
    \P((X,Y)\text{ is a manipulation pair for }f)
    \ge
    C \frac{\eps^2}{n}
  \end{equation}
  where $X\in L_3^n$ is uniformly selected, and $Y$ is obtained from
  $X$ by uniformly selecting a coordinate $i\in\{1,..,n\}$ and
  resetting the $i$'th coordinate to a random preference.
\end{theorem}
Choosing $X$, $Y$ randomly as in Theorem~\ref{thm:FrKaNi}, the result
of~\cite{FrKaNi:08} implies that a manipulation pair is obtained with
non-negligible probability (at most polynomially small in $n$), and
thus a manipulation pair can be found efficiently as long as $f$ can be efficiently evaluated.
Note however that the computational problem discussed above is different from the problem considered in previous
work~\cite{BarthOrline:91,ConitzerS03b,ElkindL05,ProcacciaR06,ConitzerS06},
where the complexity studied was that of finding a beneficial
manipulation for a specific voter, given the declared preferences of all other
voters -- since~\cite{FrKaNi:08} considers only three alternatives, a
voter with access to the social choice function can easily try all
permutations of the alternatives to find a manipulation.

Corollary~\ref{cor:NeutralBound} and
Corollary~\ref{cor:refNeutralBound}, which extend the result
of~\cite{FrKaNi:08} to the case of $4$ or more alternatives, are thus
more relevant with respect to the hardness of finding a
manipulation. They imply that in the case were votes are cast
uniformly at random, a random change of preference for a random voter
will yield a beneficial manipulation with non-negligible
probability--at most polynomially small in $q$ and $n$ by
Corollary~\ref{cor:refNeutralBound}.  Thus in the setup
of~\cite{BarthOrline:91,ConitzerS03b,ElkindL05,ProcacciaR06,ConitzerS06},
with positive probability, a single voter with black-box access to $f$
can efficiently manipulate.  This implies that approach of masking
manipulations behind computational hardness cannot hide manipulations
completely.

\medskip We note that there are other (independent) extensions
of~\cite{FrKaNi:08} for more candidates. Xia and
Conitzer~\cite{Xia-Conitzer} applied the proof strategy of~\cite{FrKaNi:08} to show that for some social choice
functions with $n$ voters and a fixed number $m$ of alternatives,
starting with a uniformly random voting profile and then randomly
resetting the ranking of one of the voters yields a manipulation pair
with probability $\Omega(1/n)$.  Their proof requires a number of properties of the social choice
functions including anonymity (the social choice outcome depends only on the number of times 
each order was chosen), homogeneity (if each vote is replaced by t identical votes the
outcome remains the same), canceling out (this condition related to neutrality - it says that one can
cancel any subset of the votes which contains each order exactly once). Most importantly the results of 
Xia and Conitzer require that certain outcomes are robust (will not change if a small linear fraction of the voters
cast a specific order) and the result 
does not give bounds on the frequency of manipulations in terms of $m$, the
number of alternatives. The later point implies that the results do not have implications for the
hardness of finding a manipulation in the setup
of~\cite{BarthOrline:91,ConitzerS03b,ElkindL05,ProcacciaR06,ConitzerS06}.

We further note that Dobzinski and Procaccia
\cite{Procaccia-Dobzinski} established an
analogous result for the case of two voters and any number of
candidates, under a comparably weak assumption on the voting rule.

\subsection{Techniques}

The result of~\cite{FrKaNi:08} are obtained by mixing combinatorial
techniques with discrete harmonic analysis. In contrast, our techniques are purely geometric and combinatorial.
In particular, we apply a variant of the
a canonical path method to prove isoperimetric bounds of "second order".
These allow to establish the existence of a large interface where $3$
bodies touch. \question{4 bodies?}
As far as we know, our result is the first one to establish such a bound in any context.

\paragraph{The canonical path method.} Before describing our
techniques, we briefly recall the canonical path
method~\cite{JerrumSinclair:90}. Given a graph $G$ and a subset $A$ of its vertices, a general
approach to proving a lower bound on the 'surface area' of
$A$---namely the number of vertices in $A$ that are attached by an
edge to a vertex outside of $A$---is as follows: for each pair $x,y$
of vertices in $G$ such that $x\in A$ and $y\not \in A$, determine a
path in $G$ between them, called the canonical path between $x$ and
$y$. Since $x$ is in $A$ and $y$ is not, there is at least one surface
vertex on each canonical path. So if one manages to prove that each
surface vertex lies on at most $r$ canonical paths, it immediately
follows that the surface of $A$ contains at least $\frac{|A|\cdot|\bar
  A|}{r}$ vertices, giving the required lower bound on the surface
area of $A$.

\paragraph{Manipulation paths.} Think of the graph $G$ having the set
$L_q^n$ of all ranking profiles as the vertex set, where the pair
$(x,y)$ is an edge if $x$ and $y$ differ on at most one coordinate. A
social choice function $f\colon L_q^n \rightarrow [q]$ naturally
partitions the vertices of $G$ into $q$ subsets. Our main interest is not
in the surface area of these subsets, however, but in the number of
manipulation points.


Our approach in the proof of Theorem~\ref{thm:neutralBound} is
therefore the following: we consider four subsets
$\inverse{A}$, $\inverse{B}$, $\inverse{C}$ and $\inverse{D}$, where the outcome is $A,B,C$ and $D$ respectively.
 We first use elementary methods to show that many edges in our graph lie
on the interface between $\inverse{A}$ and $\inverse{B}$, namely have
one vertex from each of the subsets. Similarly, many edges must lie on
the interface between $\inverse{C}$ and $\inverse{D}$.

We then define a so called \emph{manipulation path} for each pair of
edges consisting of one edge on the interface between $\inverse{A}$
and $\inverse{B}$, and one on the interface between $\inverse{C}$ and
$\inverse{D}$. The path (of edges) has the property that it either stays in one interface or the other.
If a path "transitions" from the interface between $\inverse{A}$
and $\inverse{B}$ and the interface between $\inverse{C}$ and $\inverse{D}$ then around the transition point
the function must obtain at least $3$ values. This realization allows us to apply the original
Gibbard-Satterthwaite theorem and associate a manipulation point with the path. Much of the work is then devoted to bounding the number of paths that can correspond to each manipulation point.


\paragraph{A refined geometry.}  To obtain the improved parameters of
Theorem~\ref{thm:refNeutralBound} we use a proof scheme similar to
that of Theorem~\ref{thm:neutralBound}, however we use an underlying
graph with a different edge structure. Instead of connecting every
pair $x,y\in L_q^n$ of ranking profiles that differ in just one
coordinate, we connect $x$ and $y$ only if in the coordinate $i$ in
which they differ, $y_i$ can be obtained from $x_i$ by a single
transposition. In the case where $n=1$  this is the graph that's studied in the analysis of the adjacent transposition
card shuffling~\cite{Aldous:83,Wilson:04}. The proof of the refined result requires to show that geometric and combinatorial quantities such as boundaries and manipulation points are roughly the same in the refined graph as in the original graph on $L_q^n$. This proof requires the development of a number of techniques, in particular the study of canonical paths under group actions.


\subsection{Organization of the paper}

In Section~\ref{sec:notation} we set some notations, definitions, and
some general observations. We prove Theorem~\ref{thm:neutralBound} in
Sections~\ref{sec:boundaries}, \ref{sec:first-constr-manip} and
\ref{sec:manip-points-first}. Theorem~\ref{thm:refNeutralBound} is
proved in Sections~\ref{sec:canon-paths-group},
\ref{sec:refined-boundaries}, and \ref{sec:refin-constr-manip}.
Finally, some open problems appear in
Section~\ref{sec:disc-open-problem}.

\section{Setup and notation}\label{sec:notation}

\paragraph{Rankings.} We denote by $L_q$ the set of rankings of $q$
alternatives. An element $x \in L_q$ is a permutation of the set
$[q]$. The elements ranked at top by $x$ is $x(1)$, the second is
$x(2)$ etc. Given another element $y \in L_q$, their composition $y x$
is the ranking where the element ranked at the top is $y(x(1))$ etc.

More generally we will also sometimes use $L_S$ to denote the set of rankings of a set $S$.
\todo{Rankings,total orderings, permutations. Fix notation!}

\begin{definition}[neutral social choice functions]\label{def:neutrality}
  Let $f \colon L_q^n \rightarrow [q]$ be a social choice function. We
  say that $f$ is {\em neutral} if for every $x\in L_q^n$ and every
  $y\in L_q$, $y(f(x))=f(yx_1,\ldots,yx_n)$. Informally $f$ is neutral
  if the names of the alternatives do not matter when applying $f$.
\end{definition}

\paragraph{Influences and Variance.} We call a function $f \colon
L_q^n \rightarrow [q]$ a \emph{social choice
  function} and define the \emph{influence} of the $i$:th coordinate on $f$ as
$\Inf_i(f) = \P(f(X) \neq f(X^{(i)}))$ where $X$ is uniform on $L_q^n$ and
$X^{(i)}$ is obtained from $X$ by re-randomizing the $i$:th coordinate.
Similarly we define the influence of the $i$:th coordinate w.r.t. to a single
alternative $a \in [q]$ or a pair of alternatives $a,b \in [q]$ as
\[
\Inf_i^{a}(f) = \P(f(X) = a, f(X^{(i)}) \neq a)
\]
and
\[
\Inf_i^{a,b}(f) = \P(f(X) = a, f(X^{(i)})=b)
\]
respectively.

We also define the total influence of $f$ as $\Inf (f) = \sum_{i=1}^n
\Inf_i(f)$. The following relationship is obvious, \begin{proposition}
  For any $f \colon L_q^n \rightarrow [q]$,
  \begin{equation}
    \Inf_i (f)
    =
    \sum_{a=1}^q \Inf_i^a (f)
    =
    \sum_{a,b \in [q] : a \neq b} \Inf_i^{a,b} (f)
  \end{equation}
\end{proposition}

The following standard proposition bounds the total influence with
respect to a given candidate from below by the variance with respect to that
candidate.
\note{Previously we referred to Boolean Fourier analysis, but this is
  not a Boolean function.}
\begin{proposition}
  \label{prop:sumInfVarBound}
  For any $f \colon L_q^n \rightarrow [q]$ and $a \in [q]$,
  \begin{equation}
    \sum_{i=1}^n \Inf_i^a (f)
    \ge
    \Var [ 1_{\{f(X) = a\}} ]
  \end{equation}
  where $X \in L_q^n$ is uniformly selected.
\end{proposition}
\begin{proof}
  Create a random walk
  $X=X^{(0)}, \ldots, X^{(n)}=Y$
  from $X$ by re-randomizing the $i$:th coordinate in the $i$:th step,
  i.e. for $i \in [n]$,
  $X^{(i)} \in L_q^n$ is obtained by re-randomizing the $i$:th coordinate
  of $X^{(i-1)}$.
  Letting $g(x)=1_{\{f(x) = a\}}$ and using that $X,Y$ are independent
  and that if $g(X) \neq g(Y)$
  then the value of $g$ has to change at some edge on the path we have
  \begin{eqnarray*}
    2 \Var [ 1_{\{f(X) = a\}} ]
    &=&
    2 \Var g(X)
    =
    \P(g(X) \neq g(Y))
    \le
    \\
    &\le&
    \P(\cup_{i \in [n]} \{g(X^{(i-1)}) \neq g(X^{(i)})\})
    \le
    \sum_{i=1}^n 2 \Inf_i^a(f)
  \end{eqnarray*}
\end{proof}

Further, if a function is far from all constants all such variances cannot be small:
\begin{lemma}
  \label{lem:constDist}
  For any $f \colon L_q^n \rightarrow [q]$,
  \begin{equation}
    \Dist(f, \CONST)
    \le
    \frac{q}{2} \sum_{a=1}^q
    \Var [ 1_{\{f(X) = a\}} ]
  \end{equation}
\end{lemma}
\begin{proof}
  For $a \in [q]$, let $\mu_a = \P(f(X)=a)$
  and assume w.l.o.g. that $\mu_1 \ge \mu_2 \ge \ldots \ge \mu_q$.
  Then,
  \begin{align*}
    \Dist(f, \CONST)
    & =
    (1-\mu_1)
    \le
    q \mu_1 (1-\mu_1)
    =
    \frac{q}{2} \left(1 - \mu_1^2 - (1-\mu_1)^2\right)
    \le
    \\
    & \le
    \frac{q}{2} \left(1- \sum_{a=1}^q \mu_a^2\right)
    =
    \frac{q}{2} \sum_{a=1}^q \mu_a - \mu_a^2
    =
    \frac{q}{2} \sum_{a=1}^q \Var [ 1_{\{f(X) = a\}} ]
  \end{align*}
\end{proof}

\section{Boundaries}\label{sec:boundaries}

\begin{lemma} \label{lem:boundaries1}
  Fix $q \ge 3$ and
  $f \colon L_q^n \rightarrow [q]$
  satisfying $\Dist(f, \NONMANIP) \ge \epsilon$.
  \question{Does it hold if the distance to \emph{non-manipulable} functions is $ \ge \eps$?}
  Then there exist distinct $i,j \in [n]$
  and $\{a,b\},\{c,d\} \subseteq [q]$ such that $c \notin \{a,b\}$ and
  \begin{equation}
    \Inf_i^{a,b} (f) \ge \frac{2\eps}{n q^2 (q-1)}
    \text{ and }
    \Inf_j^{c,d} (f) \ge \frac{2\eps}{n q^2 (q-1)}
  \end{equation}
\end{lemma}
\begin{proof} \note{Fixed plenty of bugs here.}
  For $a \neq b$ let
  $
  A^{a,b} =
  \left\{
    i\in [n] \mid \Inf_i^{a,b} \ge \frac{2\eps}{n q^2 (q-1)}
  \right\}
  $.

  We first claim that for all $\{a,b\}$ there exists $\{c,d\}$ such that $\{c,d\}
  \neq \{a,b\}$ and $A^{c,d} \neq \emptyset$.
  Note that $f$ being $\eps$-far from taking two values asserts that we can find
  a $c \notin \{a,b\}$
  such that $1-\frac{\eps}{q} \ge \P(f(X) = c) \ge \frac{\eps}{q-2} \ge \frac{\eps}{q}$.
  But then, by Proposition~\ref{prop:sumInfVarBound},
  \begin{equation*}
    \sum_{d \neq c} \sum_{i=1}^n \Inf_i^{c,d} (f)
    =
    \sum_{i=1}^n \Inf_i^c (f)
    \ge
    \Var [ 1_{\{f(X) = c\}} ]
    \ge
    \frac{\eps(1-\eps/q)}{q}
    \ge
    \frac{\eps (q-1)}{q^2}
  \end{equation*}
  hence there must exist some $d \neq c$ and $i \in [n]$ such that
  $\Inf_i^{c,d} \ge \frac{\eps}{n q^2} \ge \frac{2\eps}{n q^2(q-1)}$,
  and thus $A^{c,d} \neq \emptyset$.

  We next claim that
  \begin{equation}
    \label{eq:unionContains}
    |\cup_{a,b} A^{a,b}| \ge 2
  \end{equation}
  To see this, assume the contrary, i.e.
  $\cup_{a,b} A^{a,b} \subseteq \{i\}$ for some $i \in [n]$.
  Then for all $j \neq i$ it holds that
  \begin{equation}
    \label{eq:nonInflInflBound}
    \Inf_j (f) = \sum_{c,d} \Inf_j^{c,d}(f) < \frac{q(q-1)}{2}
    \frac{2\eps}{n q^2 (q-1)} = \frac{\eps}{n q}
  \end{equation}
  For $\sigma \in L_q$,
  let $f_\sigma(x) = f(x_1, \ldots, x_{i-1}, \sigma, x_{i+1}, \ldots, x_n)$
  and note that for $j \neq i$,
  \begin{equation}
    \Inf_j(f) = \frac{1}{q!} \sum_{\sigma \in L_q} \Inf_j(f_\sigma)
  \end{equation}
  while $\Inf_i(f_\sigma) = 0$.
  Hence, by \eqref{eq:nonInflInflBound}, we have
  \begin{equation*}
    \eps
    >
    q \sum_{j \neq i} \Inf_j (f)
    =
    \frac{q}{q!} \sum_{j=1}^n \sum_{\sigma } \Inf_j(f_\sigma)
    \ge
    \frac{2}{q!} \sum_{\sigma} \Dist(f_\sigma, \CONST)
    =
    2 \Dist(f,\DICT_i)
  \end{equation*}
  where the second inequality follows from Lemma~\ref{lem:constDist}
  and Proposition~\ref{prop:sumInfVarBound}.
  But this means that $f$ is $\eps/2$-close to a dictator, contradicting
  the assumption that $\Dist(f,\NONMANIP) \ge \eps$.

  Hence \eqref{eq:unionContains} holds.
  Therefore we can either find $i \neq j$
  and $\{a,b\} \neq \{c,d\}$ such that $i \in A^{a,b}$ and $j \in A^{c,d}$
  which proves the theorem,
  or we must have $|A^{a,b}| \ge 2$ for some $\{a,b\}$ while
  $A^{c,d} = \emptyset$ for any $\{c,d\} \neq \{a,b\}$. However, this
  contradicts the first claim in the proof. The result follows.
\end{proof}

As a simple corollary we have that assuming neutrality and $q\ge 4$ we may
assume $a,b,c,d$ are all distinct,
\begin{corollary}
  \label{cor:neutralPairs}
  Fix $q \ge 4$ and
  suppose $f \colon L_q^n \rightarrow [q]$ is neutral and
  satisfies $\Dist(f, \DICT) \ge \epsilon$.
  Then there exist distinct $i,j \in [n]$
  and distinct $a,b,c,d \in [q]$ such that
  \begin{equation}
    \Inf_i^{a,b} (f) \ge \frac{\eps}{n q^2 (q-1)}
    \text{ and }
    \Inf_j^{c,d} (f) \ge \frac{\eps}{n q^2 (q-1)}
  \end{equation}
\end{corollary}
\begin{proof}
  Neutrality of $f$ implies that $f$ is $1-2/q \ge 1/2$ far from the
  set of functions taking at most $2$ values.
  Since $\eps \le 1$ it follows that
  $\Dist(f, \NONMANIP) \ge \eps/2$
  \note{$\Dist(f, \NONMANIP) \ge \Dist(f, \DICT)$ does not hold.}
  Moreover, by neutrality,
  $\Inf_i^{a,b}$ does not depend on $\{a,b\}$ so
  we can choose $\{a,b\}$ and $\{c,d\}$ non-intersecting.
\end{proof}

\section{First Construction of Manipulation Paths}
\label{sec:first-constr-manip}

Similar to the definition of influence, let us now define $f$'s boundary in the $i$:th direction w.r.t. the alternatives $a,b \in [q]$ as
\[
B_i^{a,b} (f) = \{(x,y) \mid f(x)=a, f(y)=b, \forall j \neq i:x_j=y_j\}
\]

The main idea of the proof is to define a canonical path between every pair of
points on $B_i^{a,b}$ and every pair of points on $B_j^{c,d}$ in a way such that
each canonical path passes through a manipulation point while making sure that
no manipulation point can be passed by too many canonical paths. We call the
paths so constructed manipulation paths.

Let us start with defining the canonical paths in terms of one voter. The main
intuition behind the canonical paths is that in order to remain on $B_i^{a,b}$
we require that we change rankings without changing the relative order of $a$
and $b$.
\todo{The intuition should be better explained.}
Similarly, in order to remain on $B_j^{c,d}$ we require that we change
the ranking without changing the relative order of $c$ and $d$.

We now define the graph that we are working with:

\begin{definition}
  The {\em voting graph} is the graph whose vertex set is $L_q^n$ and whose
  edges are of the form $x,y$ where $x_j=y_j$ for all $j \neq i$ and $x_i \neq
  y_i$.
\end{definition}

We begin our definition of a canonical path by considering the case of one
voter.

\begin{definition} \label{def:sim_canon}
  Fix $q\ge 4$ and
  distinct $a,b,c,d \in [q]$.
  Then {\em the canonical path between $x \in L_q$ and $z \in L_q$}
  is $x, y, z$ where $y$ is obtained from $z$ by swapping $a$ and $b$ if
  necessary in order to assure that $a$ and $b$ are in the same order as in $x$.
  This first step is called a \emph{Type I move} while the second step from $y$
  to $z$ is called a \emph{Type II move}.
\end{definition}
\question{Def. does not depend on c,d. Call i path relative to a?}





Note that Type I moves preserve the order of $a$ and $b$ while Type II moves
preserve the order of $c$ and $d$. We can now define the manipulation paths used
in the first proof. These paths go from points in $B_i^{a,b}$ to $B_j^{c,d}$. To
simplify notation we assume that $i=n-1$ and $j=n$. The path is of length $2 n$
and is defined by first making all type I moves and then making all type II
moves.

\begin{definition} \label{def:canon1}
  Let $f \colon L_q^n \rightarrow [q]$,
  $(x,x') \in B_{n-1}^{a,b}$ and
  $(z,z') \in B_n^{c,d}$,
  for distinct $a,b,c,d \in [q]$.
  Then the canonical path $\Gamma$ between $(x,x')$ and $(z,z')$ is
 \begin{equation*}
    (x,x')=(x^{(0)},x'^{(0)}), \ldots, (x^{(n-2)},x'^{(n-2)}) ,
    (z^{(n-2)},z'^{(n-2)}) , \ldots, (z^{(0)},z'^{(0)}) = (z,z'),
  \end{equation*}
  where only coordinate $k$ is updated at the $k$:th first step and the $k$:th last step, i.e. for all $k$ and
  all $s \neq k$:
  \[
 (x^{(k-1)}_s,
  x'^{(k-1)}_s) =
  (x^{(k)}_s,
  x'^{(k)}_s), \quad
  (z^{(k-1)}_s,z'^{(k-1)}_s) = (z^{(k)}_s,z'^{(k)}_s),
  \]
  and
  \[
 x_k=x_k^{(k-1)} \quad , \quad x_k^{(k)}=z_k^{(k)} \quad , \quad z_k^{(k-1)}=z_k
  \]
  \[
 x'_k=x'^{(k-1)}_k \quad , \quad x'^{(k)}_k=z'^{(k)}_k \quad , \quad z'^{(k-1)}_k=z'_k
  \]
  are the canonical paths in Definition~\ref{def:sim_canon}.

\end{definition}

\section{Manipulation Points and First Proof}
\label{sec:manip-points-first}

\begin{lemma}
  \label{lem:inverseImageBound}
  For any $f \colon L_q^n \rightarrow [q]$,
  distinct $i,j \in [n]$ and
  distinct $a,b,c,d \in [q]$
  there exists a mapping
  $h \colon B_i^{a,b}(f) \times B_j^{c,d}(f) \rightarrow M$
  where
  \[
  M=\{x \in L_q^n \mid f \text{ is manipulable at } x\}
  \]
  such that for any $x \in M$
  \begin{equation}
   |h^{-1} (x) | \le 2n(q!)^{n+4}.
  \end{equation}
\end{lemma}
\begin{proof}
  Without loss of generality, let $i=n-1$ and $j=n$.
  Fix $(x,x') \in B_i^{a,b}$ and
  $(z,z') \in B_j^{c,d}$.
  Any edge on the canonical path between $(x,x')$ and $(z,z')$
  connects two pairs of points.
  The left-most pair takes the values $(a,b)$ since $f(x)=a$ and $f(x')=b$ while
  the right-most pair takes the values $(c,d)$.
 We claim that somewhere on the path there will be an edge
  $(u,u'),(v,v')$ such that either
  \begin{enumerate}
  \item[I.]
   at least one of $u,u',v,v'$ is a manipulation point.
  \item[II.]
   $f$ takes on at least three values on the points $u,u',v,v'$.
  \end{enumerate}
 To see this note that at least one of three things must happen:
  \begin{enumerate}
  \item
    Somewhere along the first half of the path the values
    of the pair changes from $(a,b)$ to something else.
   If the first value changes to $b$ then $f(x^{(k)})=a$ and
    $f(x^{(k+1)})=b$, but since the order of $a,b$ are preserved under
    Type I moves either $x^{(k)}$ or $x^{(k+1)}$ must be a manipulation point.
    A similar logic applies when the second value changes to $a$.
   Otherwise, one of the values are not in $\{a,b\}$ and therefore
    f takes on at least three values on the two pairs of this edge.
  \item
    Somewhere along the second half
    of the path - starting from the end - the values of the pair
    changes from $(c,d)$ to something else.
    If the first value changes to $d$ or the second value changes to $c$ we have
    a manipulation point since the order of $c,d$ are
    preserved under Type II moves.
    Otherwise, one of the values are not in $\{c,d\}$.
 \item
    The middle edge
    $(x^{(n-2)},x'^{(n-2)}) , (z^{(n-2)},z'^{(n-2)})$
    connects a pair with values $(a,b)$ and a pair
    with values $(c,d)$.
  \end{enumerate}

 Let $(u,u'),(v,v')$ be the first edge where one of I. or II. holds
  and note that $u,u',v,v'$ agree in all but
  two coordinates, either $\{n-1,k\}$, $\{n,k\}$ or $\{n,n-1\}$
  depending on whether the edge $(u,u'),(v,v')$
  is on the first part of the path, the second part or is the middle edge.

  We now claim that we can find a manipulation point $y$
  such that $u,u',v,v'$ and $y$ agree in all but two coordinates.
  We will let $h((x,x'),(z,z'))$ be this $y$.

  For case I. this is obvious and we can let $y$ be the any of
  $u,u',v,v$ which is a manipulation point.

  For case II., by applying the Gibbard-Satterthwaite theorem
  (Th. \ref{thm:GS}) on the restriction of $f$ to the two coordinates
  on which $u,u',v,v'$ differ we can identify a
  manipulation point $y \in L_q^n$ which only differ from $u,u',v,v'$
  on these two coordinates and also is a manipulation point of the
  original function $f$ (if there is more than one possible
  manipulation point we can just pick say the lexicographically
  smallest one).

  It remains to count the number of inverses of a manipulation point
  $y$ associated with the edge $(u,u'),(v,v')$ which can be
  any of the $2n-3$ edges of the canonical path.
  Given the edge number and $y$, there are only $(q!)^2$ possibilities
  for $u$.
  Given $u$ and the edge number there are only
  $(q!)^n$ possibilities for $x$ and $z$.
  To see this note that for each $k \in [n]$ we must have either
  \begin{itemize}
  \item
   $u_k=x_k$. In this case there are $q!$ possibilities for $z_k$.
  \item
   $u_k=z_k$. In this case there are $q!$ possibilities for $x_k$.
  \item
   $x_k,u_k,z_k$ is the canonical path from
    Definition~\ref{def:sim_canon} between $x_k$ and $z_k$.
    Then there are $\frac{q!}{2}$ possibilities for $x_k$ and $2$
    possibilities for $z_k$.
  \end{itemize}
 Finally, given $x$ and $z$ there are at most $(q!)^2$ possibilities
  for $x'$ and $z'$.
  Overall we have:
  \begin{equation}
   |h^{-1} (y) | \le (2n-3) (q!)^{n+4}
  \end{equation}
\end{proof}

\begin{proof}[Proof of Theorem \ref{thm:neutralBound}]
  By Corollary \ref{cor:neutralPairs} we can find
  distinct $i,j \in [n]$ and
  distinct $a,b,c,d \in [q]$
  such that
  \begin{equation}
    |B_i^{a,b} (f)| \ge \frac{\eps}{n q^2 (q-1)} (q!)^{n+1}
    \text{ and }
    |B_j^{c,d} (f)| \ge \frac{\eps}{n q^2 (q-1)} (q!)^{n+1}
  \end{equation}
 Applying Lemma \ref{lem:inverseImageBound} we see that
  \begin{equation}
    |M|
    \ge
   \frac{|B_i^{a,b} (f) \times B_j^{c,d} (f)|}{2n (q!)^{n+4}}
    \ge
   \frac{\eps^2}{2n^3 q^4 (q-1)^2 (q!)^2} (q!)^n
    \ge
   \frac{\eps^2}{2n^3 q^6(q!)^2} (q!)^n
  \end{equation}
  Hence,
  \begin{equation}
    \P(f\text{ is manipulable at } X)
    \ge
   \frac{\eps^2}{2n^3 q^6(q!)^2}
  \end{equation}
\end{proof}

\section{Canonical Paths and Group
  Actions}\label{sec:canon-paths-group}

In order to derive the more refined result, we will need to consider in more
detail the properties of the permutation group $L_q$ with respect to adjacent
transpositions. Again we use canonical paths arguments. We state the
arguments in a more general setup.

\begin{definition}
Let $L$ be a set.
\begin{itemize}
\item Let $P_L(\ell)$ denote the set of paths of length at most $\ell$ in $L$
  and $P_L=\cup_{\ell \in \mathbb{N}} P_L(l)$ the set of paths of finite length.
\item
  Let $L_1,L_2 \subseteq L$.
  A {\em canonical path map} on $L$  from $L_1$ to $L_2$ of {\em length} $\ell$ is a map $\Gamma \colon L_1 \times L_2 \to
  P_L(\ell)$ which satisfies that $\Gamma(x,y)$ begins at $x$ and ends at $y$
  for all $(x,y) \in L_1 \times L_2$.
\item
  Given a canonical path map $\Gamma \colon L_1 \times L_2 \to
  P_L(\ell)$
  and $0 \leq i \le \ell$ we define the inverse image mapping of the
  $i$'th vertex, $\Gamma_i^{-1}:L \rightarrow 2^{L_1 \times L_2}$ as
  \[ \Gamma_i^{-1}(z) = \{ (x,y)
  \mid \length(\Gamma(x,y)) \ge i, \Gamma(x,y)_i = z\}. \]
  Further, we let
  \begin{equation*}
    \Gamma^{-1}(z) = \cup_{i=0}^{\ell} \Gamma_i^{-1}(z)
  \end{equation*}
\item Given a
  group $H$ acting on $L$ we say that a canonical path map
  $\Gamma \colon L_1 \times L_2 \to P_L(\ell)$ is $H$-{\em invariant} if
  $H L_1 = L_1$ and $H L_2 = L_2$ and \[ \Gamma(h x, h y) = h \Gamma(x, y), \]
  for all $h \in H$ and all $(x,y) \in L_1 \times L_2$.
\end{itemize}
\end{definition}

We will use the following proposition. Recall that a group $H$ acting on $L$ is
called {\em fixed-point-free} if for all $x \in L$ and all $h \in H$ different than the
identity it holds that $h x \neq x$.

\begin{proposition} \label{prop:canon_sym}
Let $H$ be a fixed-point-free group acting on $L$ and let $\Gamma \colon L_1 \times L_2 \to P_L(\ell)$ be a canonical path map that is $H$-invariant.
Then for all $z\in L$ and $0\le i \le l$ it holds that
\begin{equation} \label{eq:canon_sym_i}
  |\Gamma^{-1}_i(z)| \le \frac{|L_1| |L_2|} {|H|}
\end{equation}
and
\begin{equation} \label{eq:canon_sym}
  |\Gamma^{-1}(z)| \leq \frac{(\ell+1) |L_1| |L_2|} {|H|}
\end{equation}
\end{proposition}

\begin{proof}
  Note that for all $i$, \[
  |L_1 \times L_2| \ge \sum_{w} |\Gamma^{-1}_i(w)| = \sum_{h \in
    H} |\Gamma^{-1}_i(h z)| = |H| |\Gamma^{-1}_i(z)|, \] where the
  first inequality follows since the value of the $i$'th vertex partitions the set of paths of length at least $i$,
  the first equality since $H$ is fixed-point-free, and the final equality from the path being $H$-invariant.
  We thus obtain: \[
  |\Gamma^{-1}(z)| \leq \sum_{i=0}^{\ell} |\Gamma_i^{-1}(z)| \leq
  \frac{(\ell+1) |L_1| |L_2|} {|H|}, \] as needed.
\end{proof}

Two applications of the result above will be given for adjacent transpositions.
\begin{definition}
 Given two elements $a,b \in [q]$ the {\em adjacent transposition} $\adj a b$
  between them is defined as follows. If $x \in L_q$ has $a$ and $b$ adjacent,
  then $\adj{a}{b} x$ is obtained from $x$ be exchanging $a$ and $b$. Otherwise,
  $\adj{a}{b} x = x$.

  We let $T$ denote the set of all $q(q-1)/2$ adjacent transpositions. Given $z
  \in T$, we define \begin{eqnarray}
    \Inf_i^{a,b ; z}(f) &=& \P(f(X) = a, f(X^{(i)})=b)
    \\
    \Inf_i^{a ; z}(f) &=& \P(f(X) = a, f(X^{(i)}) \neq a)
    \\
  \Inf_i^{a,b ; T}(f) &=& \sum_{z \in T} \Inf_i^{a,b;z}(f)
\end{eqnarray}
where $X^{(i)}$ is obtained from $X$ by re-randomizing the $i$:th coordinate $X_i$ in the
following way: with probability $1/2$ we keep it as $X_i$ and otherwise we replace
it by $zX_i$.

Finally for $x \in L_q^n$ we will let $\adji{a}{b}{i} x$ denote the element
obtained by applying $\adj{a}{b}$ on the $i$:th coordinate of $x$ while leaving
all other coordinates unchanged. \end{definition}

\begin{proposition} \label{prop:canon1}
  There exists a canonical path map
  $\Gamma \colon L_q \times L_q \to P_{L_q}(\ell)$
  of length at most $\ell = q(q-1)/2 < q^2/2$, all of whose edges are
  adjacent transpositions such that for all $z$ it holds that:
  \begin{equation}
    \label{eq:canon1}
    |\Gamma^{-1}(z)| \leq \frac{q^2 q!}{2}
  \end{equation}
\end{proposition}

\begin{proof}
  Given $x,y \in L_q$ consider the following canonical path starting at $x$ and
  ending at $y$. Take the element $y(1)$ ranked at the top for $y$ and bubble it
  to the top by performing adjacent transpositions. Then take the element $y(2)$
  ranked second for $y$ and bubble it to the second position etc. Clearly the
 length of the path is at most $q(q-1)/2$.
  Let $H=\{x \mapsto px \mid p \in L_q\}$ be the group of compositions with
  all possible permutations of the candidates.
  Since $H$ is a fixed-point-free group acting on $L_q$ and the described canonical path map is $H$-invariant
  the result follows from Proposition~\ref{prop:canon_sym}. \end{proof}

\begin{corollary}~\label{cor:sumInfVarBound2}
 For any $f \colon L_q^n \rightarrow [q]$, $a \in [q]$ and $i \in [n]$ it holds that
  \begin{equation}
   \sum_{z \in T} \Inf_i^{a; z}(f) \geq \frac{1}{q^2} \Inf_i^a(f),
  \end{equation}
  where $T$ is the set of all adjacent transpositions.
\end{corollary}

\begin{proof}
  This is a standard canonical path argument. Since both sides of the desired
  inequality involve averaging over all coordinates but the $i$'th coordinate,
  it follows that it suffices
  to prove the claim in the case where $i=n=1$.
  Let $B = \{(u,v) \in L_q \times L_q \mid f(u) = a \neq f(v), \exists z \in T: v = z u \}$ and
  note that
  \begin{equation}
    \label{eq:sumInfVarBound2a}
    \sum_{z \in T} \Inf_1^{a; z}(f) = \frac{|B|}{2q!},
  \end{equation}
  Consider the canonical path map $\Gamma$ constructed in Proposition~\ref{prop:canon1}.
  Note that each canonical path
  between an element in $A:=\{x\in L_q \mid f(x)=a\}$ and an element in $A^c$ must pass via one of the edges in $B$.
  Define $h:A\times A^C \to B$ by letting $h(x,y)$ be the first edge in $B$ which $\Gamma(x,y)$ passes through.
  Then by \eqref{eq:canon1}, for any $(u,v) \in B$,
  \begin{equation}
    |h^{-1}((u,v))| \le |\Gamma^{-1}(u)| \le \frac{q^2 q!}{2}
  \end{equation}
  Thus
  \begin{equation}
    \label{eq:sumInfVarBound2c}
    |B| \geq \frac{|A| |A^c|}{q^2 q! / 2}
  \end{equation}
  Combining \eqref{eq:sumInfVarBound2a} and \eqref{eq:sumInfVarBound2c} we obtain:
  \begin{equation*}
    \sum_{z \in T} \Inf_1^{a; z}(f) \geq \frac{1}{2q!}
    \frac{|A| |A^c|}{q^2 q! / 2} = \frac{1}{q^2} \frac{|A|}{q!}
    \frac{|A^c|}{q!} =
    \frac{1}{q^2} \Inf_1^a(f)
  \end{equation*}
\end{proof}

A second application of Proposition~\ref{prop:canon1} is the following.

\begin{proposition} \label{prop:canon2}
  Fix two elements $a, b \in [q]$ and let
  $B \subseteq L_q$ denote the set of all permutations where $a$ is ranked above
  $b$. Then there exists a canonical path map $\Gamma:B\times B \to P_{B}(q^2)$
  consisting of adjacent transpositions such that
  all permutations along the path satisfy that $a$ is ranked above $b$. Moreover
  for all $z$ it holds that:
  \[
  |\Gamma^{-1}(z)|
  \leq
  q^4 q!
  \]
\end{proposition}

\begin{proof}
$\Gamma(x,y)$ is defined as follows.
We look at all elements different than $a,b$, starting with the top one of $y$, and bubble each of them upwards to its position in $y$ ignoring $a,b$.
After we have done so, we have all elements but $a,b$ ordered as in $y$, followed by $a$, followed by $b$. We now bubble $a$ to its location in $y$ and then bubble $b$.
Note that the length of the path so defined is at most
\[
\frac{q(q-1)}{2} + 2 (q-1) = \frac{(q+4)(q-1)}{2} < q^2
\]
The proof now follows from Proposition~\ref{prop:canon_sym} by considering
the group $H$ which acts by permuting arbitrary all elements but those labeled by $a$ and $b$:
\begin{equation*}
  |\Gamma^{-1}(z)|
  \le \frac{q^2 |B|^2}{|H|}
  = \frac{q^2 (q!/2)^2}{(q-2)!}
  \le q^4 q!
\end{equation*}
\end{proof}

\section{Refined Boundaries}\label{sec:refined-boundaries}

Similarly to the previous construction we now define
the $i$:th $a$-$b$ boundary with respect to an adjacent swap $z \in T$ as
\[
B_i^{a,b ; z} (f) = \{(x,y) \mid f(x)=a, f(y)=b, x_i = z y_i, \forall j \neq i:x_j=y_j\},
\]
and the boundary with respect to arbitrary adjacent swaps on the $i$:th coordinate as
\[
B_i^{a,b ; T} (f) = \bigcup_{z \in T} B_i^{a,b;z} (f)
\]
Note that for $a \neq b$,
\begin{equation}\label{eq:inf2Boundary}
  \Inf_i^{a,b;z} (f)
  = \frac{1}{2} \P(f(X)=a,f(zX)=b)
  = \frac{1}{2} \frac{|B_i^{a,b;z}(f)|}{(q!)^n}
\end{equation}

\subsection{Manipulation points on refined boundaries}
The following two lemmas identify manipulation points on these boundaries.

\begin{lemma}
  \label{lem:nonManipBoundary}
  Fix $f \colon L_q^n \rightarrow [q]$,
  distinct $a,b \in [q]$
  and $(x,y) \in B_i^{a,b ; T}$.
  Then either $x_i=\adj{a}{b} y_i$
  or
  one of $x$ and $y$ is a $2$-manipulation point for $f$. \end{lemma} \begin{proof}
  Suppose $x_i=\adj{c}{d} y_i$ where $\{c,d\} \neq \{a,b\}$.
  Then an adjacent transposition of $c$ and $d$ will not change the order of $a$
  and $b$. Hence
  $
  b \stackrel{x_i}{>} a
  \text{ iff }
  b \stackrel{y_i}{>} a
  $.
  But then either
  i) $f(y) =b \stackrel{x_i}{>} a=f(x)$ and $x$ is a 2-manipulation point
  or
  ii) $f(x)=a \stackrel{y_i}{>} b=f(y)$ and $y$ is a 2-manipulation point.
\end{proof}

\begin{lemma}
  \label{lem:nonManipTriple}
  Fix $f \colon L_q^n \rightarrow [q]$
  and points
  $x,y,z \in L_q^n$ such that
  $(x,y) \in B_i^{a,b; T}$
  $(z,y) \in B_j^{c,b; T}$
  where $a,b,c$ are distinct and $i \neq j$.
  Then there exists a $3$ - manipulation point $w \in L_q^n$ for $f$
  such that $w_k=y_k$ for $k \notin \{i,j\}$
  and $w_i$ is equal to $x_i$ or $y_i$ except that the position of $c$ may be
  shifted arbitrarily
  and $w_j$ is equal to $z_j$ or $y_j$ except that the position of $a$ may be
  shifted arbitrarily. \end{lemma} \begin{proof}
  By Lemma~\ref{lem:nonManipBoundary} we must have
  $x_i = \adj{a}{b} y_i$
  and
  $z_j = \adj{c}{b} y_j$,
  or $x$, $y$ or $z$ is a 2-manipulation point in which case we are done.

  Now create a new triple $(x',y',z')$ by starting from $(x,y,z)$ and
  simultaneously in the $i$:th coordinate of $x$, $y$ and $z$, 
  bubbling $c$ towards the pair $ab$ until it becomes adjacent to the pair.
  Since $c$ is never swapped with $a$ or $b$ during this process
  Lemma~\ref{lem:nonManipBoundary} implies that
  for any intermediate triple $(\wt{x}, \wt{y}, \wt{z})$ we have
  $f(\wt{x})=a$, $f(\wt{y})=b$ and $f(\wt{z}) \notin \{a,b\}$,
  or one of $\wt x$, $\wt y$ and $\wt z$ is a 2-manipulation point.
  But since we also have $\wt{z} = \adj{c}{b}_j \wt{y}$, we must actually have
  $f(\wt{z}) = c$, or either $\wt y$ or $\wt z$ is a 2-manipulation point.

  Similarly bubbling $a$ towards the pair $bc$ in coordinate $j$
  starting from $(x',y',z')$
  gives us $x'',y'',z''$ all having $a,b,c$ adjacent in
  coordinates $i$ and $j$ such that
  $(x'',y'') \in B_i^{a,b; [a:b]}$ and
  $(z'',y'') \in B_j^{c,b; [c:b]}$.
  Note that $x'', y'', z''$ are equal except for a reordering of the blocks containing $a,b,c$ in coordinates $i$ and $j$.

  Now arbitrary adjacent swapping of $a,b,c$ in these coordinates of
  $x'',y''$ and $z''$ will keep the value of $f$ in $\{a,b,c\}$,
  or give rise to a 2-manipulation point by Lemma~\ref{lem:nonManipBoundary}.
  Thus we can define a social choice function with 2 voters and 3 candidates
  $f':L_{\{a,b,c\}}^2 \rightarrow \{a,b,c\}$ by letting $f'(v)=f(g(v))$, where
  $g(v) \in L_q^n$ is obtained from $x''$ by simply reordering the two blocks of elements $a,b,c$ in coordinates $i$ and $j$ to match $v_1$ and $v_2$, respectively.
  Since $f'$ takes three values and is not a dictator,
  Gibbard-Satterthwaite (Theorem~\ref{thm:GS}) implies that
  $f'$ has a manipulation point and hence
  $f$ has a 3-manipulation point satisfying our requirements.
\end{proof}

\subsection{Large Refined Boundaries}

Now we possess the right tools to prove the analogue of
Lemma~\ref{lem:boundaries1} for refined boundaries.

\begin{lemma} \label{lem:boundaries2}
  Fix $q \ge 3$ and
  $f \colon L_q^n \rightarrow [q]$
  satisfying $\Dist(f, \NONMANIP) \ge \epsilon$.
  Let $X$ be uniformly selected from $L_q^n$.
  Then either,
  \begin{equation}
    \label{eq:neutralPairs3ManipProb}
    \P(f \text{ is 2-manipulable at } X) \ge \frac{4\eps}{n q^7}
  \end{equation}
  or there exist distinct $i,j \in [n]$
  and $\{a,b\},\{c,d\} \subseteq [q]$ such that $c \notin \{a,b\}$ and
  \begin{equation}
   \Inf_i^{a,b;[a:b]} (f) \ge \frac{2\eps}{n q^7}
    \text{ and }
   \Inf_j^{c,d;[c:d]} (f) \ge \frac{2\eps}{n q^7},
  \end{equation}
  \todo{Sharpened to $q^7$ from $q^8$. Fix applications.}
\end{lemma}
\begin{proof}
  First, suppose that $\Inf_i^{a,b;z} \ge \frac{2\eps}{nq^7}$
  for some $i$, $a \neq b$ and $z \neq [a:b]$.
  Then by Lemma~\ref{lem:nonManipBoundary}
  for any point $(x,x') \in B_i^{a,b;z}(f)$
  at least one of $x$ or $x'=zx$ is a $2$-manipulation point.
  Let $\wt{M}$ be the set of all such 2-manipulation points.
  Then
  \begin{equation}
   |\wt{M}| \ge |B_i^{a,b;z}(f)|
    =
    2 (q!)^n \Inf_i^{a,b;z} (f)
    \ge
    \frac{4\eps}{n q^7} (q!)^n
  \end{equation}
  Dividing with $(q!)^n$ gives \eqref{eq:neutralPairs3ManipProb}.
  Thus, for the remainder of the proof we may assume that
  \begin{equation}
    \label{eq:boundariesAssump}
    \Inf_i^{a,b;z} < \frac{2\eps}{nq^7} \quad,\quad
    \forall i\in [n], \{a,b\} \subseteq [q], z \neq [a:b]
  \end{equation}
  Now, for $a \neq b$ let
  $
  A^{a,b} =
  \left\{
    i\in [n] \mid \Inf_i^{a,b;[a:b]} \ge \frac{2\eps}{n q^7}
  \right\}
  $.

  We first claim that for all $\{a,b\}$ there exists $\{c,d\}$ such that $\{c,d\} \neq \{a,b\}$ and $A^{c,d} \neq \emptyset$.
  Note that $f$ being $\eps$-far from taking two values asserts that we can find
  a $c \notin \{a,b\}$
  such that $1-\frac{\eps}{q} \ge \P(f(X) = c) \ge \frac{\eps}{q-2} \ge \frac{\eps}{q}$.
  But then, by Corollary~\ref{cor:sumInfVarBound2} and Proposition~\ref{prop:sumInfVarBound},
  \begin{equation*}
    \sum_{w\in T} \sum_{d \neq c} \sum_{i=1}^n \Inf_i^{c,d;w} (f)
    =
    \sum_{w\in T} \sum_{i=1}^n \Inf_i^{c;w} (f)
    \ge
    \frac{1}{q^2} \Var [ 1_{\{f(X) = c\}} ]
    \ge
    \frac{\eps (q-1)}{q^4}
  \end{equation*}
  hence there must exist some $w\in T$, $d \neq c$ and $i \in [n]$ such that
  $\Inf_i^{c,d;w} \ge \frac{\eps}{n q^6}$.
  But by \eqref{eq:boundariesAssump} we must have $w=[c:d]$, hence
  $A^{c,d} \neq \emptyset$.

  We next claim that
  \begin{equation}
    \label{eq:unionContains2}
    |\cup_{a,b} A^{a,b}| \ge 2
  \end{equation}
  To see this, assume the contrary, i.e.
  $\cup_{a,b} A^{a,b} \subseteq \{i\}$ for some $i \in [n]$.
  Then, by Corollary~\ref{cor:sumInfVarBound2}, for all $j \neq i$ it holds that
  \begin{equation}
    \label{eq:nonInflInflBound2}
    \Inf_j (f)
    \le q^2 \sum_{z\in T} \sum_a \Inf_j^{a;z}(f)
    = q^2 \sum_{z\in T, a, b>a} \Inf_j^{a,b;z}(f)
    \le \frac{q^6}{2} \frac{2\eps}{nq^7}
    = \frac{\eps}{nq}
  \end{equation}
  For $\sigma \in L_q$,
  let $f_\sigma(x) = f(x_1, \ldots, x_{i-1}, \sigma, x_{i+1}, \ldots, x_n)$
  and note that for $j \neq i$,
  \begin{equation}
    \Inf_j(f) = \frac{1}{q!} \sum_{\sigma \in L_q} \Inf_j(f_\sigma)
  \end{equation}
  while $\Inf_i(f_\sigma) = 0$.
  Hence, by \eqref{eq:nonInflInflBound2}, we have
  \begin{equation*}
    \eps
    \ge
    q \sum_{j \neq i} \Inf_j (f)
    =
    \frac{q}{q!} \sum_{j=1}^n \sum_{\sigma } \Inf_j(f_\sigma)
    \ge
    \frac{2}{q!} \sum_{\sigma} \Dist(f_\sigma, \CONST)
    =
    2 \Dist(f,\DICT_i)
  \end{equation*}
  where the second inequality follows from Lemma~\ref{lem:constDist}
  and Proposition~\ref{prop:sumInfVarBound}.
  But this means that $f$ is $\eps/2$-close to a dictator, contradicting
  the assumption that $\Dist(f,\NONMANIP) \ge \eps$.

  Hence \eqref{eq:unionContains2} holds.
  Therefore we can either find $i \neq j$ and
  $\{a,b\} \neq \{c,d\}$ such that $i \in A^{a,b}$ and $j \in A^{c,d}$
  which proves the theorem,
  or we must have $|A^{a,b}| \ge 2$ for some $\{a,b\}$ while
  $A^{c,d} = \emptyset$ for any $\{c,d\} \neq \{a,b\}$. However, this
  contradicts the first claim in the proof. The result follows.
\end{proof}

As a corollary we have that assuming neutrality and $q\ge 4$ we may
assume $a,b,c,d$ are all distinct,
\begin{corollary}
  \label{cor:neutralPairs2}
  Fix $q \ge 4$ and
  suppose $f \colon L_q^n \rightarrow [q]$ is neutral and
  satisfies $\Dist(f, \DICT) \ge \epsilon$.
  Let $X$ be uniformly selected from $L_q^n$.
  Then either,
  \begin{equation}
   \label{eq:neutralPairs2ManipProb}
    \P(f \text{ is 2-manipulable at } X) \ge \frac{2\eps}{n q^7}
  \end{equation}
  or there exist distinct $i,j \in [n]$
  and distinct $a,b,c,d \in [q]$ such that
  \begin{equation}
   \Inf_i^{a,b;[a:b]} (f) \ge \frac{\eps}{n q^7}
    \text{ and }
   \Inf_j^{c,d;[c:d]} (f) \ge \frac{\eps}{n q^7},
  \end{equation}
\end{corollary}
\begin{proof}
  Neutrality of $f$ implies that $\Dist(f,\NONMANIP) \ge \eps/2$
  \question{Maybe this needs explanation?}
  and that
  $\Inf_i^{a,b}$ does not depend on $\{a,b\}$ so
  we can choose $\{a,b\}$ and $\{c,d\}$ non-intersecting.
\end{proof}

\section{Refined Construction of Manipulation
  Paths}\label{sec:refin-constr-manip}

We now present the second construction of manipulation paths. In this
construction edges along the path will consist of adjacent transpositions
instead of general permutations as in the previous construction.
Again we construct manipulation paths between every edge on
$B_i^{a,b;\adj{a}{b}}$ and every edge on $B_j^{c,d;\adj{c}{d}}$ in a way such
that each canonical path passes through (or ``close'' to) a manipulation point while making sure
that no manipulation point can be passed by too many canonical paths. We call
the paths so constructed {\em refined manipulation paths}. The main goal in the
current construction compared to the previous one is to have better dependency
on $q$, i.e. the number of inverse images of each manipulation point should be
$\poly(n) \poly(q) q!$ instead of $2n(q!)^4 q!$ as in the previous
construction.



Let us first give two canonical paths on single coordinates that will
be used as building blocks when constructing the refined canonical paths:

\begin{proposition} \label{prop:canon4}
Fix four elements $a,b,c,d \in [q]$.
Then there exists a canonical path map
$\Gamma\colon L_q \times L_q \to P_{L_q}(q^2+2q)$ with the following properties:
\begin{itemize}
\item $\Gamma$ is a concatenation of two paths $\I$ and $\Pi$. \item
  The edges in $\I$ are arbitrary adjacent
  transpositions except $[a:b]$, thus keeping the order of $a$ and $b$ fixed.
\item
  The edges in $\Pi$ are arbitrary adjacent
  transpositions except $[c:d]$, thus keeping the order of $c$ and $d$ fixed.
\item
  For every $y\in L_q$
  there are exactly $q!$ pairs $(x,z)\in L_q\times L_q$ for which the last
  vertex of $\I$ (first vertex of $\Pi$) in the path $\Gamma(x,z)$ is equal to $y$.
\item
  For all $y \in L_q$ and $i \ge 0$ we have
  $|\Gamma_i^{-1}(y)| \le q^4 q!$
\end{itemize}
\end{proposition}

\begin{proof}
  First fix $x,z \in L_q$.
  If the order of $c$ and $d$ is the same in $x$ and $z$ then
  $\I$ has zero edges and consists only of the point $x$.
  Otherwise, $\I$ swaps the positions of $c$ and $d$ by first
  bubbling $c$ to the position of $d$ and then bubbling $d$ back to
  the original position of $c$.
  $\Pi$ is constructed as in Proposition~\ref{prop:canon2} while
  preserving the order of $c$ and $d$.

  Note that the length of $\I$ and $\Pi$ is at most
  $2q-2$ and $q^2$ respectively.
  Further, fixing the last point of $\I$ to $y$, there are two
  possibilities for $x$ and $q!/2$ possibilities for $z$.
  Hence, exactly $q!$ possible values for $(x,z)$.

  Finally, by considering
 the group $H$ which acts by permuting arbitrary all elements but
  those labeled by $a,b,c$ and $d$
  and noting that $|H| = (q-4)!$
  it follows from Proposition~\ref{prop:canon_sym} that
  \begin{equation}
    |\Gamma_i^{-1} (y)|
    \le \frac{(q!)^2}{(q-4)!}
    \le q^4 q!
  \end{equation}
\end{proof}

\begin{proposition} \label{prop:canon3}
  Fix four elements $a,b,c,d \in [q]$. Let
  \[
  X = \{ x \in L_q \mid a,b \text{ are adjacent in } x\},
  \]
  Then there exists a canonical path map $\Gamma\colon X \times L_q \to P_{L_q}(q^2+2q)$ with the following properties:
  \begin{itemize}
  \item $\Gamma$ is a concatenation of three paths $\I$, $\Delta$ and $\Pi$.
  \item
    All edges in $\I$ are adjacent transpositions
    not involving $a$ and $b$, thus keeping the \emph{rank} of $a$ and $b$ fixed.
  \item
    The edges in $\Pi$ are arbitrary adjacent
    transpositions except $[c:d]$, thus keeping the \emph{order} of $c$ and $d$ fixed.
  \item
    $\Delta$ consists of a single edge which is a reordering of a block of exactly the $4$ elements $a,b,c,d$.
  \item
    For every $y\in L_q$
    there are at most $2q^3 q!$ pairs $(x,z)\in L_q\times L_q$ for which the last
    vertex of $\I$ in the path $\Gamma(x,z)$ is equal to $y$.
    The same holds for the first vertex of $\Pi$.
  \item
    For all $y \in L_q$ and $i \ge 0$ we have
    $|\Gamma_i^{-1}(y)| \le 2q^3 q!$
  \end{itemize}
\end{proposition}

\begin{proof}
  Fix $x \in X$ and $z \in L_q$.
  The path $\I$ is constructed by first bubbling the
  element $c$ towards the block $ab$ until it is adjacent to this
  block and then doing the same with $d$.

  $\Delta$ consists of a single edge which
  reorders the block of $a,b,c$ and $d$ so that the order matches that in $z$.

  $\Pi$ is constructed as in Proposition~\ref{prop:canon2} while
  preserving the order of $c$ and $d$.

  Note that the length of $\I$ and $\Pi$ is at most
  $2q-1$ and $q^2$ respectively.

  Finally, by considering
 the group $H$ which acts by permuting arbitrary all elements but
  those labeled by $a,b,c$ and $d$
  it follows follows from Proposition~\ref{prop:canon_sym} that
  \begin{equation}
    |\Gamma_i^{-1}(y)|
    \le \frac{|X| |L_q|}{|H|}
    \le \frac{2(q-1)! q!}{(q-4)!}
    \le 2 q^3 q!
  \end{equation}
  The other properties are easy to verify.
\end{proof}

We are now ready to define the canonical path from $B_i^{a,b;[a:b]}(f)$ to
$B_j^{c,d;[c:d]}(f)$. This path is over $(L_q^n)^2$. If we only
consider the first element of each such pair, then the path can
informally be described as being constructed by concatenating three
paths $\I$, $\Delta$ and $\Pi$ where $\I$ is constructed by updating one
coordinate at a time, using the path $\I$ of Proposition~\ref{prop:canon4} for each
coordinate $k \notin \{i,j\}$, using the path $\I$ from Proposition~\ref{prop:canon3}
for coordinate $i$ and finally for coordinate $j$ using the reverse of the path $\Pi$ of
Proposition~\ref{prop:canon3} where the role of elements $a,b$ have
been interchanged with that of $c,d$.
The path $\Delta$ do the middle step from Proposition~\ref{prop:canon4} for both $i$ and $j$.
The path $\Pi$ then updates each
coordinate again using the remaining part of each path above.

\begin{proposition} \label{prop:canon5}
  Fix four distinct elements $a,b,c,d \in [q]$ and distinct
  $i,j\in[n]$. Let
  \[
  X = \{ (x,x') \in (L_q^n)^2 \mid x'=\adji{a}{b}{i} x \,,\, x' \neq x\}
  \]
  and
  \[
  Z = \{ (z,z') \in (L_q^n)^2 \mid z'=\adji{c}{d}{j} z\,,\, z' \neq z\}
  \]
  Then there exists a canonical path map
  $\overline\Gamma\colon X \times Z \to P_{(L_q^n)^2}(2n(q^2+2))$ with the following properties:
  \begin{itemize}
  \item
    $\overline\Gamma$ is a concatenation of three paths
    $\overline\I$, $\overline\Delta$ and $\overline\Pi$.
  \item
    $\overline\I$ stays in $X$ and
    for all edges $((v,v'), (w,w'))$ in $\overline\I$
    both $(v,w)$ and $(v',w')$ consist of single
    adjacent transpositions that preserve the order of $a$ and $b$ in
    each coordinate and keep the rank of $a$ and $b$ fixed in
    coordinate $i$.
  \item
    $\overline\Pi$ stays in $Z$ and
    for all edges $((v,v'), (w,w'))$ in $\overline\Pi$
    both $(v,w)$ and $(v',w')$ consist of single
    adjacent transpositions that preserve the order of $c$ and $d$ in
    each coordinate and keep the rank of $c$ and $d$ fixed in
    coordinate $j$.
  \item
    $\overline\Delta$ consists of a single edge $((v,v'),(w,w'))$
    such that $v,v',w,w'$ are all equal up to a reordering of a block of
    elements $a,b,c,d$ in coordinates $i$ and $j$.
  \item For any $(v,v') \in (L_q^n)^2$ we have
    $|\overline\Gamma^{-1}((v,v'))| \le
    7 n q^{12} (q!)^n$
  \end{itemize}
\end{proposition}

\begin{proof}
  To define $\overline\Gamma$ fix a starting pair $(x,x') \in X$ and
  an ending pair $(z,z') \in Z$.
  For this pair, the paths $\overline\I$ and $\overline\Pi$ are both constructed as a concatenation of $n$ paths:
  \begin{align}
    \label{eq:canon5ConcatPath}
    \overline\I = \overline\I(1),\ldots,\overline\I(n)
    \quad \text{and} \quad
    \overline\Pi = \overline\Pi(1),\ldots,\overline\Pi(n)
  \end{align}
  In order to define these paths first note that since $\overline\I$ must stay in
  $X$, every vertex $(v,v')$ in $\overline\I$ must satisfy $v'=[a:b]_i v$.
  Thus it is enough to describe the projection of $\overline\I$ to the first
  coordinate of each pair. Let $\I$ be this projection (so that if
  the $j$'th vertex of $\overline\I$ is $(v,v')$, then the $j$'th vertex of
  $\I$ is $v$).
  Similarily since $\overline\Pi$ must stay in $Z$, every vertex $(v,v')$ in $\overline\Pi$
  satisfies $v'=[c:d]_j v$ and it is enough to describe $\Pi$ - the projection
  of $\overline\Pi$ to the first coordinate of each pair.

  Now, for any path $\Gamma = (u(0), \ldots, u(\ell)) \in P_{L_q^n}$ let
  $\Gamma_k = (u_k(0), \ldots, u_k(\ell))$
  denote its restriction to coordinate k.
  The projections $\I$ and $\Pi$ can then be defined as follows,
  \begin{itemize}
  \item
    For any $k=1,\ldots,n\!-\!1$
    the last vertex of $\I(k)$ is equal to
    the first vertex of $\I(k+1)$,
    and the last vertex of $\Pi(k)$ is equal to
    the first vertex of $\Pi(k+1)$.
  \item
    $\forall k, m \neq k: \I_m(k)$ and $\Pi_m(k)$ are constant paths, i.e. $\I(k)$ and $\Pi(k)$ only change in coordinate $k$.
  \item
    $\forall k \notin \{i,j\}: \I_k(k)$ and
    $\Pi_k(k)$ are the paths $\I$ and $\Pi$ making up
    $\Gamma(x_k,z_k)$ in Proposition~\ref{prop:canon4}.
  \item
    $\I_i(i)$ and $\Pi_i(i)$ are the paths $\I$ and $\Pi$ making up
    $\Gamma(x_i,z_i)$ in
    Proposition~\ref{prop:canon3}.
  \item
    $\I_j(j)$ and $\Pi_j(j)$ are, respectively,
    the reverse of the paths $\Pi$ and $\I$ making up
    $\Gamma(z_j,x_j)$ in Proposition~\ref{prop:canon3} with the role
    of $(a,b)$ there swapped with that of $(c,d)$.
  \end{itemize}
  Note that this uniquely determines $\overline\Delta$ as the single edge from the last vertex of $\overline\I$ to the first vertex of $\overline\Pi$.
  The three statements about the edges of $\overline\Gamma$ now follow from
  Proposition~\ref{prop:canon4} and~\ref{prop:canon3}.

  Finally, to compute $|\overline\Gamma^{-1}((v,v'))|$ for
  $(v,v') \in (L_q^n)^2$
  we need to count the number of $(x,x') \in X$ and $(z,z') \in Z$
  such that $(v,v')$ is a vertex on the path.
  Note that $|\overline\Gamma^{-1}((v,v'))|=0$ unless $(v,v') \in X$ or $(v,v') \in Z$.
  Without loss of generality assume that $(v,v') \in X$ (the argument for $(v,v') \in Z$ is symmetric).

  Then $v$ could belong to any of the $n$ paths $\I(1),\ldots, \I(n)$.
  Suppose it belongs $\I(m)$.
  No matter what $m$ is, $v$ can be any of
  at most $q^2+2q+1$ vertices on the path $\I(m)$.
  If $m \notin \{i,j\}$ then by Proposition~\ref{prop:canon4} there can be at most $q^4q!$ possibilities for $(x_m,z_m)$,
  and if $m \in \{i,j\}$ then by Proposition~\ref{prop:canon3} there can be at most $2q^3q! < q^4q!$ possibilities for $(x_m,z_m)$.
  For all other coordinates $k \neq m$ we have that $v_k$ equals either $x_k$ or the last vertex of $\I(k)$.
  In both cases there are by Proposition~\ref{prop:canon3} at most $2q^3q!$ possibilities for $(x_k,z_k)$ if $k \in \{i,j\}$,
  and by Proposition~\ref{prop:canon4} exactly $q!$ possibilities for $(x_k,z_k)$ if $k \notin \{i,j\}$ and
  Finally, since $(x,x') \in X$ and $(z,z') \in Y$ there is
  at most one possibility for $x'$ and $z'$ given $x$ and $z$.
  Hence we have,
  \begin{equation}
    |\overline\Gamma^{-1}((v,v'))|
    \le
    n (q^2+2q+1) q^4 q!
    (2 q^3 q!)^2
    (q!)^{n-3}
    \le
    7nq^{12}(q!)^n
  \end{equation}
  since $q \ge 4$.
\end{proof}

\subsection{Proof of Theorem~\ref{thm:refNeutralBound}}

Our main claim is the following

\begin{lemma}
  \label{lem:revInverseImageBound}
  For any $f \colon L_q^n \rightarrow [q]$,
  distinct $i,j \in [n]$ and
  distinct $a,b,c,d \in [q]$
  there exists a mapping
  $h \colon B_i^{a,b;[a:b]}(f) \times B_j^{c,d;[c:d]}(f) \rightarrow M$
  where
  \[
  M=\{x \in L_q^n \mid f \text{ is $4$-manipulable at } x\}
  \]
  such that for any $x \in M$
  \begin{equation}
    |h^{-1} (x) | \le 10^4 n q^{16} (q!)^n
  \end{equation}
\end{lemma}

\begin{proof}
  Fix $(x,x') \in B_i^{a,b;[a:b]}(f)$ and
  $(z,z') \in B_i^{c,d;[c:d]}(f)$.
  Then there exist a refined canonical path $\overline\Gamma=\overline\Gamma((x,x'),(z,z'))$ (being a
  concatenation of three paths $\overline\I$, $\overline\Delta$ and $\overline\Pi$) satisfying the properties of Proposition~\ref{prop:canon5}.
  We now claim the following:

  \vspace{0.1cm}
  \textit{Claim}: Somewhere on this path there will be
  a vertex $(v,v')$ such that $v$
  is close to a 4-manipulation point $y$,
  in the sense that it differs from $y$ in at most 2 coordinates,
  and in each of those two coordinates it only differs by
  a reordering of the elements $a,b,c$ and $d$
  and an arbitrary shifting of a single element in $[q]$.
  \vspace{0.1cm}

  We will take $h((x,x'),(z,z'))$ to be an arbitrary 4-manipulation point $y$
  satisfying the closeness requirement in the claim for some vertex on the path.

  Now note that along this path at least one of the following three things must
  happen:
  \begin{enumerate}
  \item
    Somewhere along the first part $\overline\I$ of the path
    there is an edge $((v,v'),(w,w'))$
    such that $(f(v),f(v'))=(a,b)$ but $(f(w),f(w')) \neq (a,b)$.
  \item
    Somewhere along the second part $\overline\Pi$ of the path
    there is an edge $((v,v'),(w,w'))$
    such that $(f(v),f(v')) \neq (c,d)$ but $(f(w),f(w')) = (c,d)$.
  \item
    Let $((v,v'),(w,w'))$ be the single edge in $\overline\Delta$.
    Then $(f(v),f(v'))=(a,b)$ and $(f(w),f(w')) = (c,d)$.
  \end{enumerate}
  We argue that the claim follows in each of these cases:
  \begin{enumerate}
  \item
    If $e:=f(w) \neq a$, Lemma~\ref{lem:nonManipBoundary} implies that
    $w=[a:e]_kv$ for some $k \in [n]$
    (else $v$ or $w$ is a 2-manipulation point, yielding the claim).
    Since the order of $a$ and $b$ is preserved in all coordinates in $\overline\I$
    we must have $e \neq b$.
    Further $k \neq i$, since the rank of $a$ is preserved in coordinate $i$ in
    this part of the path.
    Thus $(v,v') \in B_i^{a,b;T}$ and $(v,w) \in B_k^{a,e;T}$
    and Lemma~\ref{lem:nonManipTriple}
    implies that there is a 3-manipulation point $y$ which only differ
    from $v,v',w$ and $w'$ in coordinates $i$ and $k$.
    Furthermore, $y_k$ is equal to $v_k$ or $w_k$ except that
    the position of $b$ may have been shifted arbitrarily,
    and $y_i$ is equal to $v_i=w_i$ or $v'_i=w'_i$ except that the
    position of $e$ may have been shifted arbitrarily.
    Thus it is either close to $v$ or $w$, in the sense of the claim.
    \todo{This argument has to be checked again. Don't we need to permute $e$ also sometimes, while shifting a?}

    The other possibility is that $e:= f(w') \neq b$, for which the
    claim follows by an analogous argument (remembering that $v$ and
    $v'$ only differ by an adjacent swap of $a,b$).
  \item
    The claim again follows analogously to the previous case.
  \item
    In this case Proposition~\ref{prop:canon5} guarantees that $v,v',w,w'$ only
    differ by a reordering of adjacent blocks of elements $a,b,c,d$ in
    coordinates $i$ and $j$.
    Thus we may define a new social choice function
    $f':L^2_{\{a,b,c,d\}} \rightarrow \{a,b,c,d\}$
    by letting $f'(u) = f(g(u))$ where $g(u) \in L_q^n$ is obtained from $v$ by
    simply reordering the two blocks of elements $a,b,c,d$ in coordinates $i$
    and $j$ so that they match $u_1$ and $u_2$ respectively.
    Note that this reordering can be done using adjacent transpositions
    involving $a,b,c$ and $d$ only. Hence by
    Lemma~\ref{lem:nonManipBoundary},
    $\forall u: f(g(u)) \in \{a,b,c,d\}$, or else one of the intermediate points under this
    reordering using adjacent transpositions must be a 2-manipulation point,
    yielding the claim.

    So we may assume that $f'$ is well-defined,
    i.e. takes values in $\{a,b,c,d\}$.
    However since $f'$ takes on all four values and is not a dictator,
    Gibbard-Satterthwaite (Theorem~\ref{thm:GS}) implies that $f'$ must have a
    manipulation point $u$ but then $g(u)$ must be a 4-manipulation point of $f$,
    proving the claim.
  \end{enumerate}
  Now fix $y \in M$. In order to count $|h^{-1}(y)|$ note that
  there can be at most $(4! q^2)^2$
  values of $v$ satisfying the closeness requirement to $y$ given in the claim.
  Given $v$ there are only $2$ possibilities for the vertex $(v,v')$ (depending on whether the vertex is in $\I$ or in $\Pi$).
  Further, by Proposition~\ref{prop:canon5} their can be at most
  $7nq^{12}(q!)^n$ canonical paths containing any specific vertex.
  Thus,
  \begin{equation}
    |h^{-1}(y)|
    \le
    2 (4!q^2)^2
    7nq^{12}(q!)^n
    \le
    10^4 n q^{16} (q!)^n
  \end{equation}
\end{proof}

\begin{proof}[Proof of Theorem \ref{thm:refNeutralBound}]
  By Corollary \ref{cor:neutralPairs2}, either we are done
  or we can find distinct $i,j \in [n]$ and
  distinct $a,b,c,d \in [q]$
  such that, by \eqref{eq:inf2Boundary},
  \begin{equation}
   |B_i^{a,b;[a:b]} (f)|
    \ge
    \frac{2\eps}{n q^7} (q!)^n
    \text{ and }
   |B_j^{c,d;[c:d]} (f)|
    \ge
    \frac{2\eps}{n q^7} (q!)^n
  \end{equation}
  Let $M=\{x \in L_q^n \mid f \text{ is $4$-manipulable at } x\}$.
  Applying Lemma \ref{lem:revInverseImageBound} we see that
  \begin{equation}
    |M|
    \ge
    \frac{|B_i^{a,b;[a:b]} (f) \times B_j^{c,d;[c:d]} (f)|}{10^4nq^{16} (q!)^n}
    \ge
    \frac{4\eps^2}{10^4n^3 q^{30}} (q!)^n
  \end{equation}
  Hence,
  \begin{equation}
    \P(f\text{ is 4-manipulable at } X)
    \ge
    \frac{\eps^2}{10^4n^3 q^{30}}
  \end{equation}
\end{proof}

\section{Open problems}\label{sec:disc-open-problem}
We list a few natural open problems that arise from our work.
\begin{itemize}
\item In Corollary~\ref{cor:refNeutralBound} we prove that a random
  pair $x$, $y$ is a manipulation point with non-negligible probability, if
  $y$ is obtained from $x$ by a random change in $4$ adjacent
  alternatives, applied to a random coordinate. For the case where $y$
  is obtained from $x$ by simply re-randomizing one of the
  coordinates, which is the one considered in~\cite{FrKaNi:08}, we
  only have a lower bound where $q!$ appears in the
  denominator (see Corollary~\ref{cor:NeutralBound}). It would be
  interesting to prove a polynomial lower bound in the latter case.
\item As is often the case with arguments involving canonical paths,
  we suspect that the parameters we obtained are not tight. It would
  be interesting to find the correct tight bounds. In particular, we
  are not even sure that the lower bound on the number of manipulation
  points must decrease with $q$---the correct bound may even increase
  as a function of $q$ for neutral functions. \item Our results, as
  well as those of~\cite{FrKaNi:08}, apply only to neutral functions.
  Can one prove a quantitative Gibbard-Satterthwaite theorem for
  non-neutral functions?

\item It would also be interesting to consider the
  Gibbard-Satterthwaite theorem quantitatively for non-uniform
  distributions over preferences.
 \end{itemize}

\section*{Acknowledgments}
The authors would like to thank Jeffrey Steif for helpful comments on an earlier draft.

\bibliographystyle{alpha}
\bibliography{gsbib1}

\begin{thebibliography}{FKN09}

\bibitem[Ald83]{Aldous:83}
D.~Aldous.
\newblock Random walks on finite groups and rapidly mixing {M}arkov chains.
\newblock In {\em Seminar on probability, {XVII}}, volume 986 of {\em Lecture
  Notes in Math.}, pages 243--297. Springer, Berlin, 1983.

\bibitem[Arr50]{Arrow:50}
K.~Arrow.
\newblock A difficulty in the theory of social welfare.
\newblock {\em J. of Political Economy}, 58:328--346, 1950.

\bibitem[Arr63]{Arrow:63}
K.~Arrow.
\newblock {\em Social choice and individual values}.
\newblock John Wiley and Sons, 1963.

\bibitem[BO91]{BarthOrline:91}
J.~Bartholdi, III and J.~Orline.
\newblock Single transferrable vote resists strategic voting.
\newblock {\em Soc. Choice Welf.}, 8(4):341--354, 1991.

\bibitem[BTT89]{BaToTr:89}
J.~Bartholdi, III, C.~A. Tovey, and M.~A. Trick.
\newblock Voting schemes for which it can be difficult to tell who won the
  election.
\newblock {\em Soc. Choice Welf.}, 6(2):157--165, 1989.

\bibitem[CS03]{ConitzerS03b}
Vincent Conitzer and Tuomas Sandholm.
\newblock Universal voting protocol tweaks to make manipulation hard.
\newblock In Georg Gottlob and Toby Walsh, editors, {\em {IJCAI}-03,
  Proceedings of the Eighteenth International Joint Conference on Artificial
  Intelligence, Acapulco, Mexico, August 9-15, 2003}, pages 781--788. Morgan
  Kaufmann, 2003.

\bibitem[CS06]{ConitzerS06}
Vincent Conitzer and Tuomas Sandholm.
\newblock Nonexistence of voting rules that are usually hard to manipulate.
\newblock In {\em AAAI}. AAAI Press, 2006.

\bibitem[DP08]{Procaccia-Dobzinski}
Shahar Dobzinski and Ariel~D. Procaccia.
\newblock Frequent manipulability of elections: The case of two voters.
\newblock In Christos~H. Papadimitriou and Shuzhong Zhang, editors, {\em
  Internet and Network Economics, 4th International Workshop, {WINE} 2008,
  Shanghai, China, December 17-20, 2008. Proceedings}, volume 5385 of {\em
  Lecture Notes in Computer Science}, pages 653--664. Springer, 2008.

\bibitem[EL05]{ElkindL05}
Edith Elkind and Helger Lipmaa.
\newblock Hybrid voting protocols and hardness of manipulation.
\newblock In Xiaotie Deng and Ding-Zhu Du, editors, {\em Algorithms and
  Computation, 16th International Symposium, {ISAAC} 2005, Sanya, Hainan,
  China, December 19-21, 2005, Proceedings}, volume 3827 of {\em Lecture Notes
  in Computer Science}, pages 206--215. Springer, 2005.

\bibitem[FKN09]{FrKaNi:08}
E.~Friedgut, G.~Kalai, and N.~Nisan.
\newblock Elections can be manipulated often.
\newblock In {\em Proceedings of the 49th Annual IEEE Symposium on Foundations
  of Computer Science (FOCS)}, pages 243--249, 2009.

\bibitem[FP10]{procaccia-survey}
Piotr Faliszewski and Ariel~D. Procaccia.
\newblock Ai's war on manipulation: Are we winning?
\newblock AI Magazine special issue on algorithmic game theory, to appear,
  2010.

\bibitem[Gib73]{Gibbard:73}
A.~Gibbard.
\newblock Manipulation of voting schemes: a general result.
\newblock {\em Econometrica}, 41(4):587{\~n}--601, 1973.

\bibitem[JS90]{JerrumSinclair:90}
M.~Jerrum and A.~Sinclair.
\newblock Polynomial-time approximation algorithms for ising model (extended
  abstract).
\newblock In {\em Automata, Languages and Programming}, pages 462--475, 1990.

\bibitem[Kel93]{Kelly93}
J.S. Kelly.
\newblock Almost all social choice rules are highly manipulable, but a few
  aren't.
\newblock {\em Social Choice and Welfare}, 10, 1993.

\bibitem[PR06]{ProcacciaR06}
Ariel~D. Procaccia and Jeffrey~S. Rosenschein.
\newblock Junta distributions and the average-case complexity of manipulating
  elections.
\newblock In Hideyuki Nakashima, Michael~P. Wellman, Gerhard Weiss, and Peter
  Stone, editors, {\em 5th International Joint Conference on Autonomous Agents
  and Multiagent Systems ({AAMAS} 2006), Hakodate, Japan, May 8-12, 2006},
  pages 497--504. ACM, 2006.

\bibitem[Sat75]{Satterthwaite:75}
M.~A. Satterthwaite.
\newblock {Strategy-proofness and Arrow's Conditions: Existence and
  Correspondence Theorems for Voting Procedures and Social Welfare Functions}.
\newblock {\em J. of Economic Theory}, 10:187--{\~n}217, 1975.

\bibitem[Wil04]{Wilson:04}
D.~B. Wilson.
\newblock Mixing times of lozenge tiling and card shuffling markov chains.
\newblock {\em Ann. Appl. Probab.}, 14(1), 2004.

\bibitem[XC08]{Xia-Conitzer}
Lirong Xia and Vincent Conitzer.
\newblock A sufficient condition for voting rules to be frequently manipulable.
\newblock In Lance Fortnow, John Riedl, and Tuomas Sandholm, editors, {\em
  Proceedings 9th {ACM} Conference on Electronic Commerce ({EC}-2008), Chicago,
  {IL}, {USA}, June 8-12, 2008}, pages 99--108. ACM, 2008.

\end{thebibliography}

\end{document}